\documentclass[11pt, reqno, a4paper]{amsart}

\usepackage{amsmath, amssymb, amsthm}
\usepackage{amscd}
\usepackage[figuresright]{rotating}

\usepackage{color}

\usepackage{etex}
\usepackage{dsfont}
\usepackage[matrix, arrow, curve, frame, arc]{xy}
\usepackage{pgf}
\usepackage{tikz-cd}
\usetikzlibrary{arrows,shapes,automata,backgrounds,petri,calc}
\usepackage[square,sort,comma,numbers]{natbib}
 \usepackage{url}
 \usepackage{hyperref}
 \usepackage[shortlabels]{enumitem} 
\usepackage{array}

\newcommand{\tikzAngleOfLine}{\tikz@AngleOfLine}
\def\tikz@AngleOfLine(#1)(#2)#3{%
\pgfmathanglebetweenpoints{%
\pgfpointanchor{#1}{center}}{%
\pgfpointanchor{#2}{center}}
\pgfmathsetmacro{#3}{\pgfmathresult}%
}

\newcommand{\End}{\operatorname{End}}
\newcommand{\Hom}{\operatorname{Hom}}

\newcommand{\Ext}{\operatorname{Ext}}
\newcommand{\Tor}{\operatorname{Tor}}
\newcommand{\add}{\!\operatorname{add}}
\newcommand{\pdim}{\operatorname{pdim}}
\newcommand{\m}{\!\operatorname{-mod}}

\newcommand{\rmod}{\operatorname{mod-}\!\!}
\newcommand{\proj}{\!\operatorname{-proj}}

\newcommand{\mi}{\mathfrak{m}}

\newcommand{\injdim}{\operatorname{idim}}
\newcommand{\domdim}{\operatorname{domdim}}

\newcommand{\codomdim}{\operatorname{codomdim}}

\newcommand{\hy}{\!\operatorname{-}\!}

\newtheorem{numberingthm}{Theorem}[section] 
\theoremstyle{definition}
\newtheorem{Def}[numberingthm]{Definition}

\theoremstyle{plain}
\newtheorem{Prop}[numberingthm]{Proposition}
\newtheorem{Theorem}[numberingthm]{Theorem}
\newtheorem{Conjecture}[numberingthm]{Conjecture}
\newtheorem{Cor}[numberingthm]{Corollary}
\newtheorem{Lemma}[numberingthm]{Lemma}
\newtheorem{Remark}[numberingthm]{Remark}

\theoremstyle{remark}

\theoremstyle{empty}
\newtheorem*{thmintroduction}{Theorem}

\begin{document}

\title{Higher Morita--Tachikawa correspondence}

\author[T. Cruz]{Tiago Cruz}
\address[Tiago Cruz]{
Max-Planck-Institut f\"ur Mathematik, Vivatsgasse 7, 53111 Bonn, Germany}
\email{tiago.cruz@mathematik.uni-stuttgart.de}


\begin{abstract}
	Important correspondences in representation theory can be regarded as restrictions of the Morita--Tachikawa correspondence. Moreover, this correspondence motivates the study of many classes of algebras like Morita algebras and gendo-symmetric algebras. Explicitly, the Morita--Tachikawa correspondence describes that endomorphism algebras of generators-cogenerators over finite-dimensional algebras are exactly the finite-dimensional algebras with dominant dimension at least two.
	
	In this paper, we introduce the concepts of quasi-generators and quasi-cogenerators which generalise generators and cogenerators, respectively.
	Using these new concepts, we present higher versions of the Morita--Tachikawa correspondence that takes into account relative dominant dimension with respect to a self-orthogonal module with arbitrary projective and injective dimension. These new versions also hold over Noetherian algebras which are finitely generated and projective over a commutative Noetherian ring.
\end{abstract}

\maketitle

\section{Introduction}
The celebrated Morita--Tachikawa correspondence connects basic finite-dimensional algebras having dominant dimension at least two with pairs consisting of a basic finite-dimensional algebra and a multiplicity-free generator-cogenerator (see for example \citep[Theorem 2]{zbMATH03248955}). 
An important specialisation of this correspondence is the famous Auslander correspondence \cite{zbMATH03517355} which marks the starting point for the use of Auslander and Reiten ideas in representation theory.
Not only the Auslander's correspondence but also many known correspondences like Iyama's higher Auslander correspondence \cite{zbMATH05132550} and Iyama--Solberg's correspondence \cite{zbMATH06833443} are also specialisations of the Morita--Tachikawa correspondence. 
 Moreover, algebras arising from these correspondences have appeared in many different areas like cluster theory, homological algebra, algebraic Lie theory to name a few (see \cite{KSX} and for instance see also \cite{CIM} and the references therein). The Morita--Tachikawa correspondence also brought interest to define and study many new classes of algebras for example: Morita algebras \cite{zbMATH06264368}, gendo-Frobenius algebras \cite{Yirtici} and gendo-symmetric algebras \cite{zbMATH05969535} as the counterparts through Morita--Tachikawa correspondence of self-injective, Frobenius and symmetric algebras, respectively. See also \cite{zbMATH06409569} for the counterparts of Gorenstein algebras under the Morita--Tachikawa correspondence. Recently, these correspondences started to make appearances also in the context of exact categories \cite{zbMATH07514694, ebrahimi2021higher, Grevstad}.

On the other hand, this kind of correspondences are intimately connected with the theory of dominant dimension. Recently, there has been growing interest in relative dominant dimension as a tool to study the existence and uniqueness of certain tilting modules (see \cite{zbMATH06685118, AT, CP} and the references therein) and to study the representation theory of endomorphism algebras of summands of a characteristic tilting module over quasi-hereditary algebras (see for example \cite{KSX, Cr2, CE}. Moreover, such developments provide evidence that the theory of relative dominant dimension is still in its infancy stage. In particular, the theory needs a relative analogue to the Morita--Tachikawa correspondence. 
\\
Going from classical dominant dimension to relative dominant dimension, allows for example to recover properties that are intrinsic to minimal Auslander--Gorenstein algebras into more general finite-dimensional  Gorenstein  algebras equipped with a self-orthogonal module having nice properties. Such pairs were called relative Auslander--Gorenstein pairs in \cite{CP}. 
But, so far, no relative Auslander correspondence is known for these pairs aside from a special case when the sum of the projective dimension with the injective dimension of the self-orthogonal module is at most one (see \cite{LiZhang}).

Therefore, we should expect the existence of a relative Morita--Tachikawa correspondence that takes into account the relative dominant dimension and at the same time can be used in future work for example to discover a relative analogue of the Iyama--Solberg's correspondence  to characterise relative Auslander--Gorenstein pairs.

In this paper, we make this expectation precise and we develop higher versions of the Morita--Tachikawa correspondence. To do this, it is essential to introduce higher versions of generators and cogenerators and to observe that the Morita--Tachikawa correspondence is in itself the intersection of two correspondences. Generators of a module category are the modules that contain all the indecomposable projective modules as direct summands, and cogenerators are their dual concept.
For each non-negative integer $n$, we can consider a higher $n$-version of generator that we propose to call $n$-quasi-generator to generalise the concept of generator of a module category. In this terminology, the generators are exactly the $0$-quasi-generators. Similarly, we can consider the dual concept $n$-quasi-cogenerator.  In the development of what is now called Morita theory, Morita observed that projectives and generators form a sort of dichotomy. Indeed Morita observed that any projective module affording a double centraliser property is a generator over its endomorphism algebra. In the same spirit, $n$-quasi-generators over finite-dimensional algebras over a field arise from the study of self-orthogonal modules with projective dimension $n$ affording a double centraliser property. More precisely, using $n$-quasi-generators instead of just generators we obtain the following correspondence:

\begin{thmintroduction}[see Theorem~\ref{newMoritatheorem}]
	Let $k$ be a field. For every non-negative integer $n$, there are bijections $\Phi$ and $\Psi$ between:
	$$\Gamma:= \left\{ \begin{array}{@{}c@{}}(B, M)\colon \begin{array}{@{}c@{}} B \text{ is a finite-dimensional algebra over $k$ } \\
			M \text{ is an $n$-quasi-generator of } \rmod B 
	\end{array} \end{array} \right\} $$ and
$$\Lambda:= \left\{ \begin{array}{@{}c@{}}(A, M)\colon \begin{array}{@{}c@{}} A \text{ is a finite-dimensional algebra over $k$, } \\
		M \text{ is a finitely generated $A$-module with projective dimension exactly $n$, }\\
		\text{without self-extensions, that is, } \Ext_A^{i>0}(M, M)=0 \text{ satisfying}\\
		M\hy\domdim A\geq 2
\end{array} \end{array} \right\} $$
	given as follows:\\
\begin{minipage}{.35\linewidth}
	\begin{align*}
		 \Gamma&\xrightarrow{\makebox[1.3cm]{$\Phi$}}  \Lambda\\
		 (B, M)&\mapsto (\End_B(M), M)
	\end{align*}
\end{minipage}\hfill and \hfill
\begin{minipage}{.35\linewidth}
	\begin{align*}
	 \Lambda&\xrightarrow{\makebox[1.3cm]{$\Psi$}}   \Gamma\\
		 (A, M)&\mapsto (\End_A(M)^{op}, M)
	\end{align*}
\end{minipage}. 
\end{thmintroduction}
Observe that $\Gamma\cap \Lambda$ is the class of $n$-tilting modules over a finite-dimensional algebra over a field. By fixing $n=0$, this intersection $\Gamma\cap \Lambda$ is the class of projective generators which are the main objects of study in Morita theory. 
The modules $M$ satisfying $\Ext_A^{i>0}(M, M)=0$ are known in the literature as self-orthogonal modules.

The correspondence in Theorem \ref{newMoritatheorem} restricts to a higher version of the classical Morita--Tachikawa correspondence when combined with its dual version. Indeed, classical Morita--Tachikawa correspondence corresponds to the case $n=m=0$ in the following correspondence.
\begin{thmintroduction}[see Theorem~\ref{RelativeMoritaTachikawa}]
	Let $k$ be a field.	For every pair of non-negative integers $n$ and $m$, there is a one-to-one correspondence between:
	$$\Gamma:= \left\{ \begin{array}{@{}c@{}}(B, M)\colon \begin{array}{@{}c@{}} B \text{ is a finite-dimensional algebra over $k$ } \\
			M \text{ is an $n$-quasi-generator of } \rmod B \\
			M \text{ is an $m$-quasi-cogenerator of } \rmod B 
	\end{array} \end{array} \right\} $$ and
	$$\Lambda:= \left\{ \begin{array}{@{}c@{}}(A, M)\colon \begin{array}{@{}c@{}} A \text{ is a finite-dimensional algebra over $k$, } \\
			M \text{ is a finitely generated $A$-module with projective dimension exactly $n$, }\\
			\text{and injective dimension exactly $m$,}\\
			\text{without self-extensions, that is, } \Ext_A^{i>0}(M, M)=0 \text{ satisfying}\\
			M\hy\domdim A\geq 2
	\end{array} \end{array} \right\}. $$
\end{thmintroduction}
Both Theorem \ref{RelativeMoritaTachikawa} and Theorem \ref{newMoritatheorem} are actually formulated for Noetherian algebras which are finitely generated and projective over a commutative Noetherian ring.

In the end, we explore the link between self-orthogonal quasi-generators and the homological conjectures. In particular, we prove that there exists an intermediate conjecture between the Wakamatsu tilting conjecture and the Auslander--Reiten conjecture involving self-orthogonal quasi-generators.

\section{Preliminaries}

Throughout the paper, $R$ is a commutative Noetherian ring with identity and $A$ is a \textbf{projective Noetherian $R$-algebra}, that is, an associative $R$-algebra so that $A$ is finitely generated projective as $R$-module.
By $A\m$ we denote the category of finitely generated left $A$-modules and by $\rmod A$ the category of finitely generated right $A$-modules. We write $A\proj$ to denote the full subcategory of $A\m$ whose modules are projective.
Given $M\in A\m$ (or $M\in \rmod A$) we denote by $\add_A M$ the full subcategory of $A\m$ (resp. $\rmod A$) whose modules are direct summands of direct sums of $M$ over $A$. 
By $A^{op}$ we mean the opposite algebra of $A$ and by $D$ the \textbf{standard duality functor} $\Hom_R(-, R)\colon A\m\rightarrow A^{op}\m$. 
Given $M\in A\m$, we denote by $\End_A(M)$ the endomorphism algebra of $M$ so that the multiplication $fg$ is the composition $f\circ g$ of $g$ and $f$ for $f, g\in \End_A(M)$. By $\pdim_A M$ we mean the projective dimension of the $M$ over $A$.
By an \textbf{$(A, R)$-exact sequence} we mean an exact sequence of $A$-modules which splits as
sequence of $R$-modules. We say that $M\in A\m$ is a \textbf{self-orthogonal modul}e if $\Ext_A^{i>0}(M, M)=0$.

\subsection{Relative dominant dimension}
Given $Q\in A\m\cap R\proj$ we say that $Q$ affords a \textbf{double centraliser property} if the canonical map of $R$-algebras $A\rightarrow \End_B(Q)$ is an isomorphism of $R$-algebras, where $B=\End_A(Q)^{op}$.

A useful tool to detect the existence of double centraliser properties is relative dominant dimension. We shall now recall the definition of relative dominant dimension and some of its technicalities following \cite{Cr2} and \cite{CRUZ2022410}.

\begin{Def}
	Let $M, Q\in A\m\cap R\proj$. The \textbf{relative dominant dimension of $M$ with respect to $Q$}, which we will denote by $Q\hy\domdim_{(A, R)}M$, is greater or equal than $n\in \mathbb{N}\cup \{0\}$ if there exists an $(A, R)$-exact sequence $0\rightarrow M\rightarrow Q_1\rightarrow \cdots \rightarrow Q_n$ with $Q_i\in \add_A Q$ which remains exact under $\Hom_A(-, Q)$. We write $Q\hy\domdim_{(A, R)} M=+\infty$ if $Q\hy\domdim_{(A, R)} M\geq n$ for every $n\in \mathbb{N}\cup \{0\}$. 
\end{Def}
In particular, if there exists an $(A, R)$-exact sequence $0\rightarrow M\rightarrow Q_1\rightarrow \cdots Q_n\rightarrow 0$ with $Q_i\in \add_A Q$ which remains exact under $\Hom_A(-, Q)$, then $Q\hy\domdim_{(A, R)} M=+\infty$. We will denote by $Q\hy\domdim (A, R)$ the value $Q\hy\domdim_{(A, R)} A$.
Analogously, we can consider the \textbf{relative codominant dimension of $M$ with respect to $Q$} by defining $Q\hy\codomdim_{(A, R)}M := DQ\hy\domdim_{(A^{op}, R)} DM$. We will simply write $Q\hy\domdim_A M$ whenever $A$ is a finite-dimensional algebra over a field $R$ and $Q\hy \domdim A$ instead of $Q\hy\domdim_A A$.

Over finite-dimensional algebras over a field, the classical dominant dimension is recovered from relative dominant dimension by fixing $Q$ to be a faithful projective-injective module. Similar to the classical case, relative dominant dimension also possesses some left-right symmetries, for instance $Q\hy\domdim_{(A, R)} A=Q\hy\domdim_{(A, R)} DA$ (see \citep[Corollary 3.1.5]{Cr2}).

A very important property of relative dominant dimension is that it can be characterised by the vanishing properties of certain $\Tor$ groups.

\begin{Theorem}\label{moduleMuellerparttwo}
 Let $Q\in A\m\cap R\proj$ so that $\Hom_A(Q, Q)\in R\proj$. Denote by $B$ the endomorphism algebra $ \End_A(Q)^{op}$. For $M\in A\m\cap R\proj$, $Q\hy\domdim_{(A, R)} M\geq n\geq 2$ if and only if the map $\chi^r_{DM}\colon \Hom_A(DQ, DM)\otimes_B DQ\rightarrow DM$, given by $f\otimes h\mapsto f(h)$, is an isomorphism and $\Tor_i^B(\Hom_A(DQ, DM), DQ)=\Tor_i^B(\Hom_A(M, Q), DQ)=0$, $1\leq i\leq n-2$.
\end{Theorem}
\begin{proof}
	See \citep[Theorem 3.1.4.]{Cr2}.
\end{proof}
In particular, the assumption on $Q$ is satisfied whenever $DQ\otimes_A Q\in R\proj$. Moreover, in such a case $B=\End_A(Q)^{op}$ has the base change property (for example see \citep[Proposition 2.3, Proposition 2.1]{CRUZ2022410}):
\begin{align*}
	S\otimes_R \End_A(Q)^{op}\simeq \End_{S\otimes_R A}(S\otimes_R Q)^{op}, 
\end{align*}as $S$-algebras, for every commutative Noetherian ring $S$ which is an $R$-algebra.
In the majority of cases (when $\End_A(Q)^{op}$ is again a projective Noetherian $R$-algebra having a base change property) and $M$ has a certain base change property with respect to $Q$, the computations of $Q\hy\domdim_{(A, R)} M$ can be reduced to computations over a finite-dimensional algebra over an algebraically closed field (see \citep[Theorem 3.2.5.]{Cr2}). This is in particular the case when $DQ\otimes_A Q$ and $DQ\otimes_AM$ both are finitely generated projective $R$-modules.

It follows from Theorem \ref{moduleMuellerparttwo} that a module $Q\in A\m\cap R\proj$ satisfying $\Hom_A(Q, Q)\in R\proj$ and $Q\hy\domdim_{(A, R)} A\geq 2$ affords a double centralizer property. Indeed, such a module $Q$ affords a double centralizer property if and only if $D\chi^r_{DA}$ is an isomorphism (see \citep[Lemma 5.1.2.]{Cr2}). In general, the map $D\chi_{DM}^r$ admits the following simplification:

\begin{Lemma}\label{lemma2dot1}
	The map $D\chi_{DM}^r$ fits into a commutative diagram 
	\begin{equation}
		\begin{tikzcd}
			DDM \arrow[r, "D\chi_{DM}^r", swap] & D(\Hom_A(DQ, DM)\otimes_B DQ)  \arrow[r, "\kappa", "\simeq"'] & \Hom_B(\Hom_A(DQ, DM), DDQ) \arrow[d, "\Hom_B(\varphi{,} DDQ)", "\simeq"'] \\
			M\arrow[u, "\omega_M", "\simeq"'] \arrow[r, "\alpha_M"] &
		  \Hom_B(\Hom_A(M, Q), Q)\arrow[r, "\Hom_B(\Hom_A(M{,} Q){,} \omega_Q)", "\simeq"', outer sep=0.75ex]  &\Hom_B(\Hom_A(M, Q), DDQ)  
		\end{tikzcd},
	\end{equation} where $\alpha_M$ is the map $M\rightarrow \Hom_B(\Hom_A(M, Q), Q)$ given by $\alpha_M(m)(f)=f(m)$ for every $m\in M$ and $f\in \Hom_A(M, Q)$. Here, $\omega$ is the natural transformation between the identity functor on $A\m\cap R\proj$ and $D^2$, $\varphi$ is the canonical map $\Hom_A(M, Q)\rightarrow \Hom_A(DQ, DM)$  and $\kappa$ is the isomorphism given by Tensor-Hom adjunction.
\end{Lemma}
\begin{proof}
	Indeed for any $m\in M$, $f\in \Hom_A(M, Q)$, $h\in DQ$ the following holds:
	\begin{align*}
		\Hom_B(\Hom_A(M, Q), \omega_Q)\circ \alpha_M(m)(f)(h)=\omega_Q(\alpha_M(m)(f))(h)= \omega_Q(f(m))(h)=h\circ f(m)
	\end{align*}and
\begin{align*}
	\Hom_B(\varphi, DDQ)\circ \kappa\circ D\chi_{DM}^r\circ \omega_M(m)(f)(h)&=\Hom_B(\varphi, DDQ)\circ \kappa(\omega_M(m)\circ \chi_{DM}^r)(f)(h)\\=\kappa(\omega_M(m)\circ \chi_{DM}^r)(\varphi(f))(h)&=\omega_M(m)\circ \chi_{DM}^r(\varphi(f)\otimes h)=\omega_M(m)(\varphi(f)(h))\\&=\varphi(f)(h)(m)=h\circ f(m).\tag*{\qedhere}
\end{align*}
For $\varphi$ being an isomorphism see for example \citep[Proposition 2.2.]{CRUZ2022410}.
\end{proof}

\subsection{Generators}

It is nowadays a classic result due to Morita that generators are projective modules over its endomorphism algebra.
\begin{Theorem}
	Let $M$ be a right $A$-module. $M$ is a generator if and only if $M$ is projective over $\End_A(M)$ and $M$ affords a double centralizer property.
\end{Theorem}
\begin{proof}
	See for example \citep[XVII, Theorem 7.1]{zbMATH01703931}.
\end{proof}

Let $M\in \rmod A\cap R\proj$ be a generator of $\rmod A$. It follows by the Morita theorem and \citep[Lemma 5.1.2, Theorem 3.1.4]{Cr2} that $M\hy\domdim_A A\geq 2$ whenever $DM\otimes_{\End_A(M)}M\in R\proj$. Fixing $B=\End_A(M)$,   Morita's theorem infers that $M$ is projective over $B$ and thus \begin{align}
	\Tor_{i>0}^B(DQ, \Hom_A(A, Q))\simeq \Tor_{i>0}^B(DQ, Q)=0.
\end{align} Thus $M\hy\domdim A =+\infty$ by \citep[Theorem 3.1.4]{Cr2}.

Another approach to see this (over finite-dimensional algebras) is the following. Assume that $A$ is a finite-dimensional algebra over a field and let $M$ be a right $A$-module which is a generator of $A$. So, minimal $\add M$-approximations always exist (see \cite{zbMATH03749214}). Further, for every $X\in A\m$ the minimal right $\add M$-approximation $M\rightarrow X$ is surjective. Indeed, the induced sequence $\Hom_A(A, M)\rightarrow \Hom_A(A, X)$ is surjective since $A\in \add M$.
Therefore, $M\hy\codomdim_A DA =+\infty$. By \citep[Corollary 3.1.5.]{Cr2}, $M\hy\domdim A=M\hy\codomdim_A DA=+\infty$.

\begin{Cor}
	Let $A$ be a projective Noetherian $R$-algebra and let $M\in A^{op}\m\cap R\proj$. If $M$ is a generator of $A$ as right module and $M\otimes_A DM\in R\proj$, then $M\hy\domdim (A, R)\geq 2$.
\end{Cor}
\begin{proof}
	By assumption,  $A\in \add M$ and so $R/\mi\otimes_R A\in \add R/\mi\otimes_R M$ for every maximal ideal $\mi$ of $R$. This means that $R/\mi\otimes_R M$ is a generator over $R/\mi\otimes_R A$. By the previous discussion, $R/\mi\otimes_R M\hy\domdim R/\mi\otimes_R A\geq 2$ for every maximal ideal $\mi$ of $R$. By \citep[Theorem 3.2.5.]{Cr2}, $M\hy\domdim (A, R)\geq 2$.
\end{proof}

\subsection{Cogenerators}
A module $M\in A\m\cap R\proj$ is called \textbf{$(A, R)$-cogenerator} if $DA\in \add M$. It is called \textbf{$(A, R)$-injective} if $M\in \add DA$. Observe that since $A$ is projective Noetherian $M$ is $(A, R)$-cogenerator if and only if $DM$ is a generator of $A^{op}$. So, all the results for generators can be dualised for $(A, R)$-cogenerators.
Observe, in particular, that Morita's theorem for $(A, R)$-cogenerators reads as follows: 
\begin{Cor}
	Let $M\in \rmod A\cap R\proj$. $M$ is an $(A, R)$-cogenerator if and only if $M$ is $(\End_A(M), R)$-injective and $M$ affords a double centralizer property.
\end{Cor}

\subsubsection{Relative injective dimension} $(A, R)$-injective modules made their first appearence in \cite{zbMATH03117632}. The $(A, R)$-injective modules which are projective over the ground ring are exactly the injective objects in $A\m\cap R\proj$. Indeed, they can be characterised in the following way:

\begin{Prop}\label{prop2dot5}
		Let $M\in A\m\cap R\proj$. The following assertions are equivalent.
	\begin{enumerate}[(a)]
		\item $M$ is $(A, R)$-injective;
		\item Every $(A, R)$-exact sequence
		$0\rightarrow M\rightarrow V\rightarrow W\rightarrow 0$ is split over $A$;
		\item The canonical homomorphism of $A$-modules $ M\hookrightarrow \Hom_R(A, M)$ %
		splits over $A$;
		\item The functor $\Hom_A(-, M)$ is exact on $(A, R)$-exact sequences;
	\item $\Ext_A^1(X, M)=0$ for every $X\in A\m\cap R\proj$.\end{enumerate}
\end{Prop}
\begin{proof}
	See for example \cite[Proposition 2.6., Proposition 2.10]{CRUZ2022410}.
\end{proof}

Given $M\in A\m\cap R\proj$, we say that $M$ has $(A, R)$-injective dimension $n$, denoted as $\injdim_{(A, R)} M$, if and only if $\Ext_A^{n+1}(N, M)=0$ for all $N\in A\m \cap R\proj$ and $\Ext_A^{n+2}(X, M)\neq 0$ for some $X\in A\m\cap R\proj$. 

By Proposition \ref{prop2dot5}, $(A, R)$-injective modules which are projective as $R$-modules are acyclic objects with respect to $\Hom_A(X, -)$ for every $X\in A\m\cap R\proj$. So, coresolutions by $(A, R)$-injective modules can be used to compute $(A, R)$-injective dimensions. Given $M\in A\m\cap R\proj$, such a coresolution can be obtained by applying $D$ to a projective resolution of $DM$.

\begin{Lemma}\label{prop2dot6}
		Let $M\in A\m\cap R\proj$. The following assertions are equivalent.
		\begin{enumerate}
			\item $\injdim_{(A, R)} M\leq n$;
			\item If $0\rightarrow M\rightarrow X_0\rightarrow X_1\rightarrow \cdots\rightarrow X_{n-1}\rightarrow M'\rightarrow 0$ is an $(A, R)$-exact sequence with all $X_i\in A\m\cap R\proj$ being $(A, R)$-injective, then $M'$ is $(A, R)$-injective. 
		\end{enumerate}
\end{Lemma}
\begin{proof}
	Assume that (b) holds. Then $M$ admits an $(A, R)$-injective coresolution of length at most $n+1$ (counting $X_0$). Since $(A, R)$-injective coresolutions are acyclic with respect to $\Hom_A(X, -)$ for all $X\in A\m\cap R\proj$ it follows that $\Ext_A^{n+1}(X, M)=0$ for all $X\in A\m\cap R\proj$. Conversely, consider an $(A, R)$-exact sequence $0\rightarrow M\rightarrow X_0\rightarrow X_1\rightarrow \cdots\rightarrow X_{n-1}\rightarrow M'\rightarrow 0$ with all $X_i\in A\m\cap R\proj$ being $(A, R)$-injective. As $M\in R\proj$ we obtain that $M'\in R\proj$. Let $X\in A\m\cap R\proj$. Using the fact that all $X_i$ are acyclic with respect to $\Hom_A(X, -)$ we get by applying $\Hom_A(X, -)$ to previous exact sequence that $0=\Ext_A^{n+1}(X, M)\simeq \Ext_A^1(X, M')$. By Proposition \ref{prop2dot5}, $M'$ is $(A, R)$-injective.
\end{proof}

\begin{Remark}\label{rmk2dot7}
	An elementary consequence of Lemma \ref{prop2dot6} is the equality $$\pdim_A M=\injdim_{(A^{op}, R)} DM$$ for every $M\in A\m\cap R\proj$.
\end{Remark}
This fact can also be deduced directly from the following result.

\begin{Lemma}\label{lemma2dot8}
	Let $M, N\in A\m\cap R\proj$. Then $\Ext_A^i(M, N)\simeq \Ext_{A^{op}}^i(DN, DM)$  and $\Tor_i^{A}(DM, N)\simeq \Tor_i^{A^{op}}(N, DM)$ for all $i\geq 0$.
\end{Lemma}
\begin{proof}
	For the first equality see \citep[Lemma 2.15]{CRUZ2022410}. For the second, observe that  \begin{align}
		\Tor_i^{A^{op}}(N, DM)=H_i(N^{\bullet}\otimes_{A^{op}}DM)\simeq H_i(DM\otimes_A N^{\bullet})=\Tor_i^A(DM, N),
	\end{align}for all $i\geq 0$, where $N^\bullet$ denotes a deleted projective left resolution of $N$ over $A$.
\end{proof}

\section{Quasi-generators}

We will now introduce quasi-generators and quasi-cogenerators to prove our main result.
\begin{Def}
	Let $M\in \rmod A\cap R\proj$ and $n$ be a non-negative integer. We say that $M$ is an \textbf{$n$-quasi-generator} of $\rmod A$ if $M\otimes_A DM\in R\proj$ and $n$ is the minimal integer so that there exists an $(A, R)$-exact sequence 
	\begin{align}
		0\rightarrow A\rightarrow M_0\rightarrow \cdots \rightarrow M_n\rightarrow 0\label{eq3}
	\end{align}which remains exact under $\Hom_A(-, M)$ and each $M_i$ belongs in $\add_A M$.
\end{Def}
Dually, we define the concept of relative quasi-cogenerator. 

\begin{Def}
	Let $M\in \rmod A\cap R\proj$ and $n$ be a non-negative integer. We say that $M$ is an \textbf{$n$-quasi-$(A, R)$-cogenerator} if $M\otimes_A DM\in R\proj$ and $n$ is the minimal integer so that there exists an $(A, R)$-exact sequence 
	\begin{align}
		0\rightarrow M_n\rightarrow \cdots \rightarrow M_0\rightarrow  DA\rightarrow 0
	\end{align}which remains exact under $\Hom_A(M, -)$ and each $M_i$ belongs in $\add_A M$.
\end{Def}
In this way, a right module $M$ is an $n$-quasi-generator of $\rmod A$ if and only if $DM$ is an $n$-quasi-$(A^{op}, R)$-cogenerator. This terminology is chosen to be consistent with the one used in \cite{CRUZ2022410}.
 Similarly, the same definitions can be considered for left modules. 

\begin{Theorem}\label{newMoritatheorem}
	Let $R$ be a commutative Noetherian ring.	For each non-negative integer $n$, there are bijections $\Phi$ and $\Psi$ between:
	$$\Gamma:= \left\{ \begin{array}{@{}c@{}}(B, M)\colon \begin{array}{@{}c@{}} B \text{ is a projective Noetherian $R$-algebra, } \\
			M \text{ is an $n$-quasi-generator of } \rmod B 
	\end{array} \end{array} \right\} $$ and
	$$\Lambda:= \left\{ \begin{array}{@{}c@{}}(A, M)\colon \begin{array}{@{}c@{}} A \text{ a projective Noetherian $R$-algebra,} \\
			M\in  A\m\cap R\proj \text{ has the following properties: } \\
			M\hy\domdim (A, R)\geq 2, \ \pdim_A M=n, \\
			DM\otimes_A M\in R\proj \text{ and } \Tor_{i>0}^A(DM, M)=0
	\end{array} \end{array} \right\} $$
	given as follows:\\
	\begin{minipage}{.35\linewidth}
		\begin{align*}
			\Gamma&\xrightarrow{\makebox[1.3cm]{$\Phi$}}  \Lambda\\
			(B, M)&\mapsto (\End_B(M), M)
		\end{align*}
	\end{minipage}\hfill and \hfill
	\begin{minipage}{.35\linewidth}
		\begin{align*}
			\Lambda&\xrightarrow{\makebox[1.3cm]{$\Psi$}}   \Gamma\\
			(A, M)&\mapsto (\End_A(M)^{op}, M)
		\end{align*}
	\end{minipage}. 
\end{Theorem}
\begin{proof}
	Let $(A, Q)\in \Lambda$. Since $\pdim_A Q=n$ there exists an exact sequence
	\begin{align}
		0\rightarrow P_n\rightarrow P_{n-1}\rightarrow \cdots\rightarrow P_0\rightarrow Q\rightarrow 0
	\end{align} with $P_i\in A\proj$. Applying $DQ\otimes_A -$ yields the exact sequence
\begin{align}
	0\rightarrow DQ\otimes_A P_n\rightarrow DQ\otimes_A P_{n-1}\rightarrow \cdots\rightarrow DQ\otimes_A P_0\rightarrow DQ\otimes_A Q\rightarrow 0 \label{eq6}
\end{align}Since $DQ\otimes_A Q\in R\proj$, (\ref{eq6}) is $(\End_A(Q)^{op}, R)$-exact and $\End_A(Q)^{op}$ is a projective Noetherian $R$-algebra. Write $B=\End_A(Q)^{op}$.
Applying $D$ to (\ref{eq6}) we obtain an $(B, R)$-exact sequence (see also \citep[Proposition 2.1.]{CRUZ2022410})
\begin{align}
	0\rightarrow B\rightarrow \Hom_A(P_0, Q)\rightarrow \cdots \rightarrow \Hom_A(P_n, Q)\rightarrow 0. \label{eq7}
\end{align} Observe that each $\Hom_A(P_i, Q)\in \add_B \Hom_A(A, Q)=\add_B Q$. To see that (\ref{eq7}) remains exact under $\Hom_B(-, Q)$, consider the following commutative diagram
\begin{equation*}
	\begin{tikzcd}[column sep = small]
		0 \arrow[r] & P_n \arrow[r] \arrow[d, "\alpha_{P_n}"] & P_{n-1}\arrow[r] \arrow[d, "\alpha_{P_{n-1}}"] & \cdots \arrow[r] & Q \arrow[r] \arrow[d, "\simeq"] & 0\\
		0 \arrow[r] & \Hom_B(\Hom_A(P_n, Q), Q)\arrow[r] & \Hom_B(\Hom_A(P_{n-1}, Q), Q)\arrow[r] & \cdots  \arrow[r] & \Hom_B(B, Q)\arrow[r] & 0 
	\end{tikzcd}
\end{equation*}Since $Q\hy\domdim (A, R)\geq 2$, the maps $\chi^r_{DP_i}$ are isomorphisms and $Q\otimes_B DQ\simeq DA\in R\proj$ by \citep[Theorem 3.1.3.]{Cr2}. By Lemma \ref{lemma2dot1}, the commutativity of the previous diagram and the exactness of the upper row implies that (\ref{eq7}) remains exact under $\Hom_B(-, Q)$. This shows that $Q$ is an $i$-quasi-generator of $\rmod B$ for some $i\leq n$. It is indeed an $n$-quasi-generator of $\rmod B$ because applying $\Hom_B(-, Q)$ to an exact sequence (\ref{eq3}) of length $i$ yields a projective resolution of $Q$ over $A$ of length $i$; which would contradict the projective dimension of $Q$ over $A$ being $n$.
Further, $Q$ affords a double centralizer property since $Q\hy\domdim (A, R)\geq 2$ (see \citep[Lemma 5.1.2.]{Cr2}), that is, $A\simeq \End_B(Q)$ as $R$-algebras. This shows that $\Phi\circ \Psi$ is the identity.

Conversely, suppose that $(B, Q)\in \Gamma$. By assumption, $Q\hy\domdim (B, R)$ is infinite and ${Q\otimes_B DQ}\in R\proj$. So, $A:=\End_B(Q)$ is a projective Noetherian $R$-algebra.  By \citep[Theorem 3.1.4]{Cr2} (for right modules), the canonical map $$DQ\otimes_A Q\simeq DQ\otimes_{\End_B(Q)}\Hom_B(DQ, DB)\rightarrow DB$$ is an isomorphism and $\Tor_{i>0}^A( DQ, Q)=\Tor_{i>0}^A(DQ, \Hom_B(DQ, DB))=0$. Thus, $DQ\otimes_A Q\in R\proj$. Further, $Q$ affords a double centralizer property, that is, $B\simeq \End_A(Q)^{op}$ as $R$-algebras by \citep[Lemma 5.1.2.]{Cr2}. So, assuming that $\Phi$ is well-defined, we get that $\Psi\circ \Phi$ is the identity. Hence, it remains to prove that $(A, Q)\in \Lambda$. By construction, since $Q\otimes_B DQ\in R\proj$ the canonical map $Q\otimes_B DQ\rightarrow D\Hom_B(Q, Q)=DA$ is an isomorphism (see \citep[Proposition 2.1.]{CRUZ2022410}). By \citep[Theorem 3.1.4.]{Cr2}, $Q\hy\domdim (A, R)\geq 2$.
It remains to prove that $\pdim_A Q=n$. By assumption, there exists an $(B, R)$-exact sequence
$0\rightarrow B\rightarrow Q_0\rightarrow \cdots\rightarrow Q_n\rightarrow 0$ which remains exact under $\Hom_B(-, Q)$. Since each $\Hom_B(Q_i, Q)$ is projective over $A$ we get that $\pdim_A Q\leq n$. If $\pdim_A \leq i<n$, then by the previous direction $Q$ would be an $i$-quasi-generator of $\rmod \End_A(Q)^{op}=\rmod B$, contradicting the minimality of $n$.
\end{proof}

Dualising Theorem \ref{newMoritatheorem} we obtain the following.

\begin{Theorem}\label{newdualMoritatheorem}
	Let $R$ be a commutative Noetherian ring.	For each non-negative integer $n$, there are bijections $\varPhi$ and $\varPsi$ between:
	$$\varGamma:= \left\{ \begin{array}{@{}c@{}}(B, M)\colon \begin{array}{@{}c@{}} B \text{ is a projective Noetherian $R$-algebra, } \\
			M \text{ is an $n$-quasi-$(B, R)$-cogenerator as right $B$-module } 
	\end{array} \end{array} \right\} $$ and
	$$\varLambda:= \left\{ \begin{array}{@{}c@{}}(A, M)\colon \begin{array}{@{}c@{}} A \text{ a projective Noetherian $R$-algebra,} \\
			M\in  A\m\cap R\proj \text{ has the following properties: } \\
			M\hy\domdim (A, R)\geq 2, \ \injdim_{(A, R)} M=n, \\
			DM\otimes_A M\in R\proj \text{ and } \Tor_{i>0}^A(DM, M)=0
	\end{array} \end{array} \right\} $$
	given as follows:\\
	\begin{minipage}{.35\linewidth}
		\begin{align*}
			\varGamma&\xrightarrow{\makebox[1.3cm]{$\varPhi$}}  \varLambda\\
			(B, M)&\mapsto (\End_B(M), M)
		\end{align*}
	\end{minipage}\hfill and \hfill
	\begin{minipage}{.35\linewidth}
		\begin{align*}
			\varLambda&\xrightarrow{\makebox[1.3cm]{$\varPsi$}}   \varGamma\\
			(A, M)&\mapsto (\End_A(M)^{op}, M)
		\end{align*}
	\end{minipage}. 
\end{Theorem}
\begin{proof}
By definition of quasi-generator, $D$ induces a bijection between $\varGamma$ and $\Gamma$ given by \linebreak${(B, M)\mapsto (B^{op}, DM)}$. Denote this bijection by $\gamma$. By Lemma \ref{lemma2dot8}, Remark \ref{rmk2dot7}, \citep[Corollary 3.1.5]{Cr2}, $D$ induces a bijection between $\varLambda$ and $\Lambda$ given by $(A, M)\mapsto (A^{op}, DM)$. Denote this bijection by $\psi$. The map $\varPhi$ is the composition $\psi^{-1}\circ \Phi\circ \gamma$ and $\varPsi$ is the composition $\gamma^{-1}\circ \Psi\circ \psi$. So, the result follows by Theorem \ref{newMoritatheorem}.
\end{proof}

Combining Theorem \ref{newdualMoritatheorem} with Theorem \ref{newMoritatheorem} we obtain:
\begin{Theorem}\label{RelativeMoritaTachikawa}
			Let $R$ be a commutative Noetherian ring.	For every pair of non-negative integers $n$ and $m$, there is a one-to-one correspondence between:
	$$\Gamma\cap \varGamma:= \left\{ \begin{array}{@{}c@{}}(B, M)\colon \begin{array}{@{}c@{}} B \text{ is a projective Noetherian $R$-algebra, }\\
			M \text{ is an $n$-quasi-generator of } \rmod B  \\
			M \text{ is an $m$-quasi-$(B, R)$-cogenerator as right $B$-module } 
	\end{array} \end{array} \right\} $$ and
	$$\Lambda\cap \varLambda:= \left\{ \begin{array}{@{}c@{}}(A, M)\colon \begin{array}{@{}c@{}} A \text{ a projective Noetherian $R$-algebra,} \\
			M\in  A\m\cap R\proj \text{ has the following properties: } \\
			M\hy\domdim (A, R)\geq 2, \ \injdim_{(A, R)} M=m, \ \pdim_A M=n \\
		DM\otimes_A M\in R\proj \text{ and } \Tor_{i>0}^A(DM, M)=0
	\end{array} \end{array} \right\}. $$
\end{Theorem}

In particular, this generalises \citep[Theorem 4.1]{CRUZ2022410} and by consequence also \citep[Theorem 2]{zbMATH03248955} whose published proof builds on the works of \cite{MR161888, MR96700}.

\section{Self-orthogonal quasi-generators}
Assume now that $A$ is finite-dimensional algebra over a field $k$.
Recall the definition of tilting module.
\begin{Def}
Given $n\in \mathbb{N}\cup \{0\}$, a module $T\in A\m$ is called an \textbf{$n$-tilting module} if it satisfies the following conditions:
\begin{enumerate}[(a)]
	\item $T$ has projective dimension at most $n$ over $A$;
	\item $T$ is self-orthogonal, that is, $\Ext_A^{i>0}(T, T)=0$;
	\item There exists an exact sequence $0\rightarrow A\rightarrow T_0\rightarrow \cdots \rightarrow T_n\rightarrow 0$ with all $T_i\in \add_A T$.
\end{enumerate}
\end{Def}
Since an $n$-tilting module is self-orthogonal, the last condition  can be replaced by saying that $T$ is an $i$-quasi-generator for some $0\leq i\leq n$. 

Conversely, what conditions are necessary and sufficient for a quasi-generator to be a tilting module? We conjecture the following:

\begin{Conjecture}\label{quasigeneratorconjecture}
	Let $A$ be a finite-dimensional algebra over a field $k$. Every self-orthogonal $n$-quasi-generator has projective dimension at most $n$.  
\end{Conjecture}

Hence, this conjecture implies that every self-orthogonal $n$-quasi-generator is actually an $n$-tilting module. Observe that if $n=0$, then Conjecture \ref{quasigeneratorconjecture} is exactly the Auslander--Reiten conjecture proposed in \cite{zbMATH03526890}. A similar conjecture is the Wakamatsu tilting conjecture originated in \cite{zbMATH04053850}.

\begin{Conjecture}[Wakamatsu tilting conjecture]
	Let $A$ be a finite-dimensional algebra over a field. Let $T\in A\m$. If $T$ is self-orthogonal with finite projective dimension such that $T\hy\domdim_A A=+\infty$, then $T$ is a tilting module.
\end{Conjecture}

We can see that the Wakamatsu tilting conjecture is actually stronger than Conjecture \ref{quasigeneratorconjecture}.

\begin{Prop}
	If the Wakamatsu tilting conjecture is true for all finite-dimensional algebras over a field, then Conjecture \ref{quasigeneratorconjecture} is also true.
\end{Prop}
\begin{proof}
	Let $M$ be a self-orthogonal $n$-quasi-generator (as left $A$-module). In particular, $0=D\Ext_A^{i>0}(M, M)=\Tor_{i>0}^A(M, DM)$. By Theorem \ref{newMoritatheorem}, $M$ is self-orthogonal and it has projective dimension exactly $n$  as $\End_A(M)$-module and $M\hy\domdim_{\End_A(M)} \End_A(M)\geq 2$. Denote by $B$ the endomorphism algebra $\End_A(M)$. Since $\End_B(M)^{op}\simeq A$ we obtain by \citep[Theorem 3.1.4]{Cr2} that $M\hy\domdim_B B=+\infty$. So $M\in B\m$ is in the conditions of the Wakamatsu tilting conjecture, and thus the Wakamatsu tilting conjecture says that $M$ is an $n$-tilting module over $B$. Hence, it is an $i$-quasi-generator of $\rmod B$ for some $i=0, \ldots, n$. By Theorem \ref{newMoritatheorem}, we obtain that $M$ has projective dimension $i$ over $A$. 
\end{proof}

Therefore, if we fix $n=0$ we obtain an alternative way to see that the Wakamatsu tilting conjecture is stronger than the Auslander--Reiten conjecture. This fact can also be seen by combining \citep[Proposition 3.1]{zbMATH05179140} with \cite{zbMATH03526890}. Therefore by abbreviating the above mentioned conjectures to their initials, their relationship is given as follows:
\begin{equation*}
	(WTC)\implies (C\ref{quasigeneratorconjecture}) \implies (ARC).
\end{equation*} 
Another interesting consequence from Conjecture \ref{quasigeneratorconjecture} is a characterisation of Gorenstein algebras in terms of quasi-generators objects. In fact, if Conjecture \ref{quasigeneratorconjecture} is true for all finite-dimensional algebras over a field, then a finite-dimensional algebra over a field $A$ is Gorenstein if and only if $DA$ is an $n$-quasi-generator for some non-negative integer $n$.

\section*{Acknowledgments}

The author would like to thank Chrysostomos Psaroudakis for his comments towards improving this manuscript.

\bibliographystyle{alphaurl}
\bibliography{ref}

\end{document}